\documentclass[12 pt, reqno]{amsart}

\usepackage{latexsym}
\usepackage{amssymb}
\usepackage{mathrsfs}
\usepackage{amsmath}
\usepackage{fancybox,color}
\usepackage{enumerate}
\usepackage[latin1]{inputenc}

\usepackage[colorlinks]{hyperref}
\usepackage{eurosym}
\usepackage{url}
\usepackage{graphicx}

\newcommand{\kom}[1]{}
 \def\1{\raisebox{2pt}{\rm{$\chi$}}}

\newtheorem{theorem}{Theorem}[section]
\newtheorem{corollary}[theorem]{Corollary}
\newtheorem{lemma}[theorem]{Lemma}
\newtheorem{proposition}[theorem]{Proposition}
\theoremstyle{definition}
\newtheorem{definition}[theorem]{Definition}

\newcommand{\R}{{\mathbb R}}

 \newcommand{\eps}{{\varepsilon}}
 \def\1{\raisebox{2pt}{\rm{$\chi$}}}
 
\newcommand{\abs}[1]{\left|#1\right|}

\newcommand{\sign}{\operatorname{sign}}

\def\vint_#1{\mathchoice%
          {\mathop{\kern 0.2em\vrule width 0.6em height 0.69678ex depth -0.58065ex
                  \kern -0.8em \intop}\nolimits_{\kern -0.4em#1}}%
          {\mathop{\kern 0.1em\vrule width 0.5em height 0.69678ex depth -0.60387ex
                  \kern -0.6em \intop}\nolimits_{#1}}%
          {\mathop{\kern 0.1em\vrule width 0.5em height 0.69678ex
              depth -0.60387ex
                  \kern -0.6em \intop}\nolimits_{#1}}%
          {\mathop{\kern 0.1em\vrule width 0.5em height 0.69678ex depth -0.60387ex
                  \kern -0.6em \intop}\nolimits_{#1}}}
\def\vintslides_#1{\mathchoice%
          {\mathop{\kern 0.1em\vrule width 0.5em height 0.697ex depth -0.581ex
                  \kern -0.6em \intop}\nolimits_{\kern -0.4em#1}}%
          {\mathop{\kern 0.1em\vrule width 0.3em height 0.697ex depth -0.604ex
                  \kern -0.4em \intop}\nolimits_{#1}}%
          {\mathop{\kern 0.1em\vrule width 0.3em height 0.697ex depth -0.604ex
                  \kern -0.4em \intop}\nolimits_{#1}}%
          {\mathop{\kern 0.1em\vrule width 0.3em height 0.697ex depth -0.604ex
                  \kern -0.4em \intop}\nolimits_{#1}}}

\newcommand{\aveint}[2]{\mathchoice%
          {\mathop{\kern 0.2em\vrule width 0.6em height 0.69678ex depth -0.58065ex
                  \kern -0.8em \intop}\nolimits_{\kern -0.45em#1}^{#2}}%
          {\mathop{\kern 0.1em\vrule width 0.5em height 0.69678ex depth -0.60387ex
                  \kern -0.6em \intop}\nolimits_{#1}^{#2}}%
          {\mathop{\kern 0.1em\vrule width 0.5em height 0.69678ex depth -0.60387ex
                  \kern -0.6em \intop}\nolimits_{#1}^{#2}}%
          {\mathop{\kern 0.1em\vrule width 0.5em height 0.69678ex depth -0.60387ex
                  \kern -0.6em \intop}\nolimits_{#1}^{#2}}}
\newcommand{\ud}{\, \text{d}}

\newcommand{\sgn}{\operatorname{sgn}}

\renewcommand{\div}{\operatorname{div}}
\newcommand{\dd}{\partial}
\renewcommand{\d}{\hspace{0.05em}\textup{d}}
\usepackage{changepage}

\makeatletter\def\SL@eqnlefttext #1{\hbox to 0pt{\kern 75 pt 
\llap{\SL@margintext{#1}\quad}\hss}}
\makeatother

\begin{document}

\title[Sign-changing uniqueness]{Uniqueness of sign-changing solutions to Trudinger's equation}
\author[Anttila]{Riku Anttila}
\address{Department of Mathematics and Statistics, University of
Jyv\"askyl\"a, FI-40014 Jyv\"askyl\"a, Finland}
\email{riku.t.anttila@jyu.fi}

\author[Lindqvist]{Peter Lindqvist}
\address{Department of Mathematical Sciences, Norwegian University of Science and Technology, NO-7491 Trondheim, Norway}
\email{peter.lindqvist@ntnu.no}

\author[Parviainen]{Mikko Parviainen}
\address{Department of Mathematics and Statistics, University of
Jyv\"askyl\"a, FI-40014 Jyv\"askyl\"a, Finland}
\email{mikko.j.parviainen@jyu.fi}

\date{\today}
 \subjclass[2020]{}
\thanks{This work was supported by the Research Council of Finland (project 360185), and the Finnish Ministry of Education and Culture's Pilot for Doctoral Programmes (Pilot project Mathematics of Sensing, Imaging and Modelling). We thank Verena B\"ogelein for providing a reference.}

\begin{abstract}
   We  establish uniqueness of sign-changing solutions to Trudinger's parabolic equation with time dependent $C^2$ Cauchy-Dirichlet boundary data.
\end{abstract}

\maketitle

\section{Introduction}
We  consider the doubly non-linear evolutionary equation of Trudinger \cite{trudinger68}, which reads
\[
    \frac{\partial}{\partial t} \left(\abs{u}^{p-2}u\right)\, =\, \div(\abs{\nabla u}^{p-2} \nabla u), \quad 1 <  p < \infty,
    \]
    in the bounded domain $\Omega_T = \Omega \times (0,T),$ where $\Omega \subset \R^n$ is a Lip\-schitz domain.
    In \cite{trudinger68} N. Trudinger observed that, due to the homogeneous structure, no 'intrinsic scaling' is needed in Harnack's inequality. This was later proved by J. Kinnunen and T. Kuusi \cite{kinnunenk07} as well as by U. Gianazza and V. Vespri \cite{gianazzav06}.
In particular, when it comes to sign-changing solutions, little is known about Trudinger's equation in contrast to the   situation for the evolutionary $p$-Laplace equation
$$  \frac{\dd u}{\dd t}\,=\, \div (|\nabla u|^{p-2}\nabla u).$$
Especially, under natural assumptions the uniqueness of sign-changing solutions is a challenging open problem for Trudinger's equation. For   solutions $u$ with $\inf\{u\} \ge  0$ this is settled, see  \cite{lindgrenl22,lindqvistps} for references. 

\sloppy
We shall focus on the uniqueness of sign-changing weak solutions with time dependent boundary data.
Our main result, Theorem~\ref{thm:Main}, establishes uniqueness for sign-changing solutions with $C^2$ boundary data. It has not escaped our attention that the $C^2$ assumption can be somewhat relaxed. Proposition~\ref{prop:monotone} covers the special case of monotone lateral boundary values without additional regularity.  For solutions with time-independent Dirichlet boundary values on the lateral boundary $\partial \Omega \times [0,T]$\, F. Otto proved uniqueness in \cite{otto96}: for example \emph{weak solutions in} $L^p(0,T;W^{1,p}_0(\Omega))$ \emph{are unique}. In fact, these solutions turn out to  have a time  derivative $\partial_t(|u|^{p-2}u)$. Less regular solutions in  $L^p(0,T;W^{1,p}_0(\Omega))$ cannot exist!

The idea of our proof is, first following Otto's approach in \cite{otto96},   to use Kru\v{z}kov's \cite{kruzkov70} doubling of time variables  to establish the comparison principle, but now extended to time dependent boundary values. The argument requires some space between the boundary values. Thus an extra device is needed:  we finish the proof by considering the auxiliary solutions
 \begin{equation*}\begin{cases}
    u_{+\varepsilon}\quad\text{with boundary values}\quad\psi + \eps\\
    u_{-\varepsilon}\quad\text{with boundary values}\quad\psi - \eps.
\end{cases}\end{equation*}
 An \emph{arbitrary} weak solution $u$ with boundary values $\psi$ will be squeezed
 $$ u_{-\varepsilon}\,\leq\,u\,\leq\, u_{+\varepsilon}.$$
 We can conclude the proof by using the $C^2$ assumption as in \cite{lindqvistps} to prove  that both auxiliary solutions converge to \emph{the same solution} as $\varepsilon \to 0$. The limit procedure is based on an estimate which has been worked out only for $p >2$. However, Sections \ref{sec:prelim} and \ref{sec:comp} are valid in the full range $1 < p < \infty$.

\section{Preliminaries}
\label{sec:prelim}
 
Let $\Omega \subseteq \R^n$ be  a bounded Lipschitz domain. We  consider the equation in the domain  $\Omega_T = \Omega \times (0,T)$, when $1  <  p < \infty$. (In Section 4 we must take $p\geq 2$.)
We use the standard notation $$
    \partial_p \Omega_T := (\overline{\Omega} \times \{0\}) \cup (\partial \Omega \times [0,T]).
    $$
    for the parabolic boundary 
and denote the \emph{parabolic Sobolev space} by $L^p(0,T;W^{1,p}(\Omega))$. Recall that this function space consists of measurable functions $u : \Omega_T  \to \R$ such that the function  $x \mapsto u(x,t)$ belongs to $W^{1,p}(\Omega)$ for almost every $t \in (0,T)$, and $u$ has the finite  norm
\[
    \biggl(\int\limits_{0}^T \!\! \int\limits_\Omega \bigl(\abs{u}^p + \abs{\nabla u}^p \bigr)\d x \d t \biggr)^{\frac{1}{p}} < \infty.
\]
Here $\nabla u$ denotes the spatial gradient.
The function space $L^p(0,T;W^{1,p}_0(\Omega))$ is defined analogously.
Define 
$$b(z) := \abs{z}^{p-2}z, \qquad  b(0) = 0.$$
We write  $z^+ = \max\{z,0\}$.

\begin{definition}\label{def:PDE}
    Let $\psi \in C(\overline{\Omega_T}) \cap L^p(0,T;W^{1,p}(\Omega))$, and\\ $u \in L^p(0,T;W^{1,p}(\Omega))$. We say that $u$ is a \emph{weak solution} with the (Dirichlet) boundary values $\psi$ if the following conditions hold.
    \begin{enumerate}
        \item $u - \psi \in L^p(0,T;W_0^{1,p}(\Omega))$.
        \vspace{3pt}
        \item For every $\gamma \in C_0^{\infty}(\Omega_T)$ we have
        \begin{equation}\label{eq:PDE}
            \int\limits_0^T \!\! \int\limits_\Omega \bigl( -b(u)\partial_t \gamma + \langle \abs{\nabla u}^{p-2}\nabla u,\nabla \gamma \rangle\bigr) \d x \d t\, =\, 0.
        \end{equation}
        \vspace{3pt}
        \item $u$ satisfies the initial value condition 
        \begin{equation}\label{eq:InitialValues}
            \lim_{t \to 0+} \int\limits_\Omega \abs{b(u(x,t)) - b(\psi(x,t))} \d x\, =\, 0.
        \end{equation}
    \end{enumerate}
\end{definition}
Given $\psi$ as above, a weak solution $u$ always exists with the given initial condition, see \cite{altl83} and  also \cite{misawan23}. By \cite{bogeleindl21} weak solutions are locally H\"{o}lder continuous. 
By Theorem 1.2  and Theorem 1.3 in \cite{bogeleindl21}, the boundary values are continuously attained in $\partial \Omega \times (0,T)$ and $\Omega \times \{0\}$.

\bigskip

{\bf Remark.}  It is customary to treat the (Cauchy) \emph{initial values} and the (Dirichlet) \emph{lateral boundary values} separately, as in \cite{altl83}, \cite{otto96}, \cite{bogeleindl21} for instance. So we require
\begin{equation*}\begin{cases}
    u(x,0) = \psi_0(x)\quad\text{when}\quad x \in \Omega\\
    u(x,t) = \psi(x,t) \quad \text{on}\quad \partial \Omega \times (0,T].
\end{cases}\end{equation*}
More precisely, conditions  (1) and (2) above are required for a bounded function $\psi$ belonging to $C(\overline{\Omega_T})\,\cap L^p(0,T;W^{1,p}(\Omega))$.
Condition (3) now reads
$$
\lim_{t \to 0+} \int\limits_\Omega \abs{b(u(x,t)) - b(\psi_0(x))} \d x\, =\, 0
$$
for some initial values $\psi_0 \in C(\overline{\Omega})$. A classical example, for instance, is $\psi \equiv 1,\,\,\psi_0 \equiv 0.$

\section{Comparison principle}\label{sec:comp}

In this section  a suitable comparison principle for the uniqueness proof is provided.
Our method is strongly influenced by Otto's work \cite{otto96}, but includes some necessary modifications to handle time-dependent boundary values.
The  method originates from the work of S. Kru\v{z}kov \cite{kruzkov70}.
See \cite[Section 4]{otto96} for further discussion.

It will be important \emph{not}  to assume  any existence of time derivatives. Indeed, under such an extra assumption 
the result is well-known. For the convenience of the reader, we present a standard proof following the notation in \cite{lindqvistps}: Consider two solutions $u_1,u_2$ that have \emph{time derivatives} $\partial_t(|u_j|^{p-2}u_j)$ at least in $L^1(\Omega_T)$ and take the same boundary values.
Let $H_{\delta}(s)$ denote the approximation 
\begin{equation*}H_{\delta}(s) \,=\, \begin{cases} 0,\,\,\,s\leq 0 \\
    \dfrac{s}{\delta},\,\,\,0 <s<\delta\\
    1,\, \,\,s\geq \delta.
\end{cases} \end{equation*}
of the Heaviside function.
We use the test  function $\zeta = H_{\delta}\bigl(u_2-u_1\bigr)$ in both equations
$$\int\limits_{t_1}^{t_2}\!\!\int\limits_{\Omega}\Bigl(\zeta\, \dd_t(|u_j|^{p-2}u_j)\,+\, \big\langle |\nabla u_j|^{p-2}\nabla u_j,\nabla \zeta \big\rangle \Bigr) \d x \d t\,=\,0, \quad j = 1,2.$$
 Subtracting the equations, we get 
 \begin{gather*}
  \int\limits_{t_1}^{t_2}\!\!\int\limits_{\Omega}\dd_t\left(|u_2(x,t)|^{p-2}u_2(x,t) -|u_1(x,t)|^{p-2}u_1(x,t)\right)H_{\delta}\bigl(u_2-u_1\bigr) \d x \d t\\
  = - \frac{1}{\delta}\int \!\!\int\! \big\langle |\nabla u_2|^{p-2}\nabla u_2 -   |\nabla u_1|^{p-2}\nabla u_1,\nabla u_2 - \nabla u_1 \big\rangle \d x \d t\,\,\leq\,\, 0,
\end{gather*}
by the vector inequality $\big\langle |b|^{p-2}b-|a|^{p-2}a,b -a \big\rangle \,\geq\,0$. Letting $\delta \to 0$ we obtain
$$ \int\limits_{t_1}^{t_2}\!\!\int\limits_{\Omega}\dd_t\left(|u_2(x,t)|^{p-2}u_2(x,t) -|u_1(x,t)|^{p-2}u_1(x,t)\right)^+\! \d x \d t\,\, \leq \,\, 0.$$
This implies
\begin{align*}
  &\int\limits_{\Omega}\left(|u_2(x,t)|^{p-2}u_2(x,t) -|u_1(x,t)|^{p-2}u_1(x,t)\right)^+\! \d x \\  \leq \,&\int\limits_{\Omega}\left(|u_2(x,t_1)|^{p-2}u_2(x,t_1) -|u_1(x,t_1)|^{p-2}u_1(x,t_1)\right)^+\! \d x \,\,\rightarrow \,\,0,
\end{align*}
when $ t_1\, \to\,0$.
We conclude that $u_2\leq u_1$ in $\Omega_T$. By symmetry  $u_1\leq u_2$.  

The same argument proves the corresponding  comparison principle for solutions with time derivatives.

Next we state the comparison principle with boundary values in $ C(\overline{\Omega_T}) \cap L^p(0,T;W^{1,p}(\Omega))$. Observe that now no existence of time derivatives for solutions is assumed. 
\begin{theorem}[\textsf{Comparison Principle}] \label{thm:Comparison-Princple}
  Let $1 < p < \infty.$  Let $u_1$ and $u_2$ be two weak solutions with boundary values $\psi_1,\psi_2$ as in Definition~\ref{def:PDE}.
    Assume that there is a constant $\varepsilon > 0$ such that
    \begin{equation}\label{eq:psi-epsilon}
    \psi_1(x,t) \leq \psi_2(x,t) - \varepsilon \quad\text{ for all }\quad (x,t) \in \partial_p \Omega_T.
\end{equation}
Then $u_1 \leq u_2$ in  $ \Omega_T$.
\end{theorem}
We introduce the tools for constructing the required test functions. Fix a  convex function $\eta \in C^\infty(\R)$ satisfying
\begin{align*}
\eta(z) :=
    \begin{cases}
        0 & \text{ when } z \leq 0\\
        z - \frac{1}{2} & \text{ when } z \geq 1.
    \end{cases}
\end{align*}
A key property  will be that, upon rescaling, the derivative $\eta'$ approximates the Heaviside function and converges to the $\sign$-function.
For a parameter $\sigma > 0$, we define the two functions
\begin{align*}
\eta_{\sigma,+}(z) := \sigma\eta\bigl( \frac{z}{\sigma} \bigr) \quad\text{ and }\quad \eta_{\sigma,-}(z) := \sigma\eta\bigl(-\frac{z}{\sigma}\bigr),    
\end{align*}
and also
\begin{equation}\label{eq:q-function}
    q_{\sigma,\pm}(z,w) := \int\limits_{w}^z \!\eta_{\sigma,\pm}'(s\, -\, w)\, b'(s) \,\d s.
\end{equation}
Here $\eta_{\sigma,\pm}'$ is the derivative of $\eta_{\sigma,\pm}$, and $b'(z) = (p-1)|z|^{p-2}$ belonging to $ L^1_{\text{loc}}(\R)$ is the distributional derivative of $b$. 
Note that $b'(z) > 0$ when $z\neq 0$.
 
  Recall that a smooth function $\R \to \R$ is convex if and only if its first derivative is non-decreasing.
    Since $\eta$ is a smooth, convex function, so is $\eta_{\sigma,\pm}$  and hence the derivative  $\eta_{\sigma,\pm}'$ is \emph{non-decreasing}. This will be crucial in the next lemma.

\begin{lemma}\label{lemma:sub-diff}
    For all $z_1,z_2,w \in \R$ and $\sigma > 0$ we have
    \[
        q_{\sigma,\pm}(z_1,w) - q_{\sigma,\pm}(z_2,w) \,\geq \,\eta_{\sigma,\pm}'(z_2 - w)\,(b(z_1) - b(z_2)). 
    \]
\end{lemma}

\begin{proof}
   We prove the claim for $\eta_{\sigma,+}$ since the proof for $\eta_{\sigma,-}$ is similar. First assume that $z_1\ge z_2$ and observe
    \begin{align*}
q_{\sigma,+}(z_1,w)- q_{\sigma,+}(z_2,w)
&=\,\,\int\limits_{z_2}^{z_1}\eta_{\sigma,+}'(s-w)\,b'(s)\ud s\\
\geq \,\,\eta_{\sigma,+}'(z_2-w) \int\limits_{z_2}^{z_1}\,b'(s)\ud s
&= \,\,\eta_{\sigma,+}'(z_2-w)\, (b(z_1)-b(z_2)).
\end{align*}
In the inequality we used the fact that $\eta_{\sigma,+}'$ is non-decreasing. If $z_1<z_2$, then the estimate reads as 
\begin{align*}
q_{\sigma,+}(z_1,w)- q_{\sigma,+}(z_2,w)&=\,\,-\int\limits_{z_1}^{z_2}\eta_{\sigma,+}'(s-w)b'(s)\ud s\\
&\,\,\ge \eta_{\sigma,+}'(z_2-w)\,(b(z_1)-b(z_2)).
\end{align*}
This proves the inequality.
\end{proof}

By convention, we often omit the spatial variable $x$ and write $u(t) = u(x,t)$ and $\nabla u(t) = \nabla u(x,t).$
Finally, let $\phi \in C_0^\infty(\R)$ be a function satisfying
\begin{equation*}
    \begin{cases}
        \phi(z) \geq 0\, \text{ for all } \,z \in \R, \\
        \phi(z) = 0 \,\text{ for all }\, \abs{z} \geq 1,\\
        \int_{\R} \phi \, \d z = 1.\\
    \end{cases}
\end{equation*}
For a parameter $\lambda > 0$ and a non-negative function $\gamma \in  C_0^\infty(0,T)$, we define the test function $\gamma_\lambda \in C_0^\infty(\R^2)$ by
\begin{equation}\label{eq:gamma-lambda}
    \gamma_\lambda(t,s) := \frac{1}{\lambda}\, \phi\Bigl( \frac{t - s}{\lambda} \Bigr)\, \gamma\Bigl( \frac{t + s}{2} \Bigr).
\end{equation}
There is no $x$-dependence here. Thus $\nabla \gamma_{\lambda} = 0.$ Later we shall send $\lambda$ to $0$.

\begin{lemma}[\textsf{Key lemma}]
\label{lemma:Key}
    Let  $\psi_1,\psi_2, u_1,u_2$ and the constant $\eps>0$ be  as in Theorem \ref{thm:Comparison-Princple}, and let  $\gamma \in C_0^\infty(0,T)$ be non-negative.
    For all $\lambda$ small enough, all $\sigma > 0$, and almost every $s \in (0,T)$, we have
    \begin{align}\label{eq:Key1}
        0 & \geq \int\limits_0^T \!\!  \int\limits_\Omega\Bigl\{-q_{\sigma,+}(u_1(t),u_2(s)) \partial_t \gamma_\lambda(t,s)\Bigr.\\
        + & \, \Bigl.\gamma_\lambda(t,s)  \eta_{\sigma,+}''(u_1(t)\! -\! u_2(s)) \langle \abs{\nabla u_1(t)}^{p-2} \nabla u_1(t) , \nabla u_1(t)\! -\! \nabla u_2(s) \rangle \!\Bigr\}\d x \d t. \nonumber
    \end{align}
    Similarly for small enough $\lambda$, all $\sigma > 0$, and almost every $t \in (0,T)$, we have
    \begin{align}\label{eq:Key2}
        0 & \geq \int\limits_0^T \!\! \int\limits_\Omega \Bigl\{-q_{\sigma,-}(u_2(s),u_1(t)) \partial_s \gamma_\lambda(t,s)\Bigr.\\
        - & \,\Bigl. \gamma_\lambda(t,s) \eta_{\sigma,+}''(u_1(t)\! -\! u_2(s)) \langle \abs{\nabla u_2(s)}^{p-2} \nabla u_2(s) , \nabla u_1(t) - \nabla u_2(s) \rangle\!\Bigr\} \d x  \d s. \nonumber
    \end{align}
\end{lemma}

\begin{proof}
    We first prove \eqref{eq:Key1}.
    By the uniform continuity, there is a $\delta >0$ such that $$|\psi_1(x,t) -\psi_1(x,s)| < \varepsilon/2 \quad \text{when} \quad  \abs{t-s} < \delta$$
    for all $t,s \in [0,T]$  and $x \in \partial \Omega.$  It follows that 
    \[
    \psi_1(x,t) - \psi_2(x,s) 
    <  -\varepsilon/2.
    \]
    By the definition of $\gamma_\lambda$, it holds that
    $$
    \gamma_\lambda(t,s) = 0\quad \text{when} \quad   \abs{t-s} \geq \lambda $$
    for all $t,s \in [0,T]$.

    Also for small enough $\lambda>0$ only depending on the support of $\gamma$, it holds that $t \mapsto \gamma_\lambda(t,s) \in C_0^\infty(0,T)$ for all $s > 0$. From now on we only consider small enough $\lambda>0$; we see  that the choice of $\lambda$ only depends on $\gamma$ and $\delta$.

Recall that $\eta_{\sigma,+}'(z) = 0$ when $z \leq 0$.
We next observe that
    \[
        f(x,t) := \eta_{\sigma,+}'(u_1(x,t) - u_2(x,s))\, \gamma_{\lambda}(t,s)
    \]
   belongs to
    $L^p(0,T;W_{0}^{1,p}(\Omega))$ for almost every $s \in (0,T)$.
    We prove that \eqref{eq:Key1} holds for such $s$.

    For notational convenience, we extend $f(t) = 0$ for $t > T$ and $t < 0$, and put  $u_1(t) = \psi_1(0)$ for $t < 0$.
    We first consider the Stekloff averages
    \[
        f_{h}(x,t) := \frac{1}{h} \int\limits_{t}^{t+h} f(x,\tau)\d \tau.
    \]
    Note that $f_h \in L^p(0,T;W_{0}^{1,p}(\Omega))$ for all small enough $h > 0$, and that $f_h$ has  the  time derivative
    \[
        \partial_t f_{h}(t) = \frac{f(t+h) - f(t)}{h} \in L^\infty(\Omega_T).
    \]
    Hence, we may use $f_{h}$ as a test function in the equation  \eqref{eq:PDE}.
    This gives
    \[
        \int\limits_0^T \!\! \int\limits_\Omega b(u_1(t))\partial_t f_{h}(t)\d x \d t \, =\, \int\limits_{0}^T \!\! \int\limits_\Omega  \langle \abs{\nabla u_1(t)}^{p-2}\nabla u_1(t) ,\nabla f_{h}(t)\rangle \d x \d t.
    \]
    We also have
    \begin{align*}
        \int\limits_{0}^T \!\! \int\limits_\Omega b(u_1(t))\partial_t f_{h}(t) \d x \d t & = \int\limits_0^T \!\! \int\limits_\Omega b(u_1(t)) \, \frac{f(t\!+\!h) - f(t)}{h} \,\d x \d t\\
        & = \int\limits_0^T \!\! \int\limits_\Omega - \,\frac{b(u_1(t)) - b(u_1(t\!-\!h))}{h}f(t)\d x \d t.
    \end{align*}
    The second equality follows from a simple change of variables and the notational conventions made earlier in the proof.
    By the definition of $f$ and Lemma \ref{lemma:sub-diff}, we estimate
    \begin{align*}
      \int\limits_0^T& \!\! \int\limits_\Omega \frac{b(u_1(t)) - b(u_1(t-h))}{h}\,f(t) \d x \d t\\
        = & \quad \int\limits_0^T \!\! \int\limits_\Omega \frac{b(u_1(t)) - b(u_1(t-h))}{h}\,\eta_{\sigma,+}'(u_1(t) \!-\! u_2(s))\gamma_\lambda(t,s)\d x \d t\\
        \geq & \quad \int\limits_0^T \!\! \int\limits_\Omega \frac{q_{\sigma,+}(u_1(t),u_2(s)) - q_{\sigma,+}(u_1(t-h),u_2(s))}{h}\,\gamma_\lambda(t,s)\d x \d t\\
        = & \quad \int\limits_0^T \!\! \int\limits_\Omega -q_{\sigma,+}(u_1(t),u_2(s))\, \frac{\gamma_\lambda(t+h,s) - \gamma_\lambda(t,s)}{h}\d x \d t.
    \end{align*}
    In the last equality we used the fact that $u_2(s)$ does not depend on $t$.
    A  combination of the previous three displays gives
    \begin{align*}
        \int\limits_0^T& \!\! \int\limits_\Omega -q_{\sigma,+}(u_1(t),u_2(s))\, \frac{\gamma_\lambda(t+h,s) - \gamma_\lambda(t,s)}{h}\d x \d t\\
        \leq & \quad 
       - \int\limits_0^T \!\! \int\limits_\Omega
        \langle \abs{\nabla u_1(t)}^{p-2}\nabla u_1(t) ,\nabla f_{h}(t)\rangle \d x \d t.
    \end{align*}
    By letting $h \to 0+$, we get
    \begin{align*}
        0 &\geq  \int\limits_0^T \!\! \int\limits_\Omega \Big\{ -q_{\sigma,+}(u_1(t),u_2(s))\, \partial_t \gamma_\lambda(t,s)\\
        &\hspace{13 em}+ \langle \abs{\nabla u_1(t)}^{p-2}\nabla u_1(t),\nabla f(t) \rangle\Big\} \d x \d t\\
        &=  \int\limits_0^T \!\! \int\limits_\Omega \Big\{-q_{\sigma,+}(u_1(t),u_2(s))\, \partial_t \gamma_\lambda(t,s)\\
         +& \gamma_\lambda(t,s) \eta_{\sigma,+}''(u_1(t)\! -\! u_2(s)) \langle \abs{\nabla u_1(t)}^{p-2}\nabla u_1(t),\nabla u_1(t) - \nabla u_2(s) \rangle\Big\}\d x \d t.
    \end{align*}
    In the integrand we used the chain rule to compute $\nabla f(t)$. This completes the proof of \eqref{eq:Key1}.
    
    As for the proof of  \eqref{eq:Key2},
    note that $\eta_{\sigma,-}'(z) = 0$ when $z \geq 0$. Hence, by a similar argument as in the beginning of the proof, we have that
    \[
        g(x,s) := \eta_{\sigma,-}'(u_2(s) - u_1(t))\,\gamma_\lambda(t,s)
    \]
    is contained in $L^p(0,T;W_0^{1,p}(\Omega))$ for all $\lambda>0$ small enough, all $\sigma > 0$ and almost every $t \in (0,T)$.
    The remainder of the proof is essentially similar as above.
\end{proof}

Next we derive fundamental estimates by sending $\sigma \to 0+$ in \eqref{eq:Key1} and \eqref{eq:Key2}.

\begin{proposition}\label{prop:sigma->0}
   Let  $\psi_1,\psi_2$ as well as $ u_1,u_2$  be  as in Theorem \ref{thm:Comparison-Princple}, and let  $\gamma \in C_0^\infty(0,T)$ be a non-negative function.
    Then for all $\lambda$ small enough we have that
    \begin{equation}\label{eq:DoublingOfVariables}
        \int\limits_0^T \!\!\! \int\limits_0^T \!\!\!\int\limits_\Omega\! \bigl(\left( b(u_1(t)) - b(u_2(s)) \right)^+(\partial_t + \partial_s)\gamma_{\lambda}(t,s)\bigr)\d x\d t\d s \,\,\geq\,\, 0.
    \end{equation}
\end{proposition}

\begin{proof}
    By integrating \eqref{eq:Key1} over $s \in (0,T)$ and \eqref{eq:Key2} over $t \in (0,T)$, and adding these two integrals together, we get
    \begin{align*}
         0\,\geq \int\limits_0^T \!\!\! \int\limits_0^T \!\!\! \int\limits_\Omega\!&\Bigl\{ -q_{\sigma,+}(u_1(t),u_2(s))\partial_t \gamma_\lambda(t,s) - q_{\sigma,-}(u_2(s),u_1(t))\partial_s \gamma_\lambda(t,s)\Bigr.\\
        & \Bigl. + \gamma_\lambda(t,s) \eta_{\sigma,+}''(u_1(t) \!-\! u_2(s)) M(\nabla u_1(t),\nabla u_2(s))\Bigr\}\, \d x \d t \d s,
    \end{align*}
    where we have written
    $$M(\nabla u_1,\nabla u_2)\,=\,
\big\langle \abs{\nabla u_1}^{p-2}\nabla u_1  - \abs{\nabla u_2}^{p-2}\nabla u_2, \nabla u_1 - \nabla u_2 \big\rangle.$$
    
    Note that the term on the second row is non-negative. This is because $\eta_{\sigma,+}'' \geq 0$ by convexity, $\gamma_{\lambda} \geq 0$ by assumption, and we have the standard inequality $\langle \abs{a}^{p-2}a - \abs{b}^{p-2}b, a - b \rangle \geq 0$.
    It follows that
    \begin{align*}
      \int\limits_0^T \!\!\!& \int\limits_0^T \!\!\! \int\limits_\Omega\! \left( q_{\sigma,+}(u_1(t),u_2(s))\partial_t \gamma_\lambda(t,s) + q_{\sigma,-}(u_2(s),u_1(t))\partial_s \gamma_\lambda(t,s)\right) \d x \d t \d s\,\\
       &\geq\, 0.  
\end{align*}
    We now prove \eqref{eq:DoublingOfVariables} by letting $\sigma \to 0$.
    To this end, we verify a suitable convergence.
    First, note that when $\sigma \to 0+$
    $$
    \eta_{\sigma,+}'(z)\, \to\, \sgn(z):=\begin{cases}
0, &z\le 0, \\
1, & z> 0,
\end{cases}
    $$ 
     uniformly in $\R \setminus [-a,a]$ where $a > 0$ is fixed.
    Hence, as $\sigma \to 0+$, we have
    \[
    q_{\sigma,+}(z,w) \to (b(z) - b(w))^+ \quad\text{ pointwise in}\quad \R^2,
    \]
    where $ (b(z) - b(w))^+=\max\{b(z) - b(w),0\}$.
    By a straightforward computation, we see that
    \[
        \abs{q_{\sigma,+}(z,w)} \leq (b(z) - b(w))^+\quad \text{ for all}\quad z,w \in \R \,\,\text{and}\,\, \sigma > 0.
    \]
    There are analogous properties for
    \[
    q_{\sigma,-}(z,w) \to (b(w) - b(z))^+ \quad\text{ and } \quad\abs{q_{\sigma,-}(z,w)} \leq (b(w) - b(z))^+.
    \]
    Hence, by letting $\sigma \to 0$, inequality \eqref{eq:DoublingOfVariables} follows from Lebesgue's dominated convergence theorem.
\end{proof}

Next, we let $\lambda \to 0$ in \eqref{eq:DoublingOfVariables} to prove the following estimate.

\begin{proposition}\label{prop:lambda->0}
   Let  $\psi_1,\psi_2$ as well as $ u_1,u_2$  be  as in Theorem \ref{thm:Comparison-Princple}, and let  $\gamma \in C_0^\infty(0,T)$ be non-negative. Then
    \begin{equation}\label{eq:lambda->0}
        \int\limits_0^T \!\! \int\limits_\Omega \left(b(u_1(t)) - b(u_2(t))\right)^+ \gamma'(t)\d x \d t \,\geq\, 0.
    \end{equation}
\end{proposition}

\begin{proof}
    We consider $\lambda>0$ small enough so that \eqref{eq:DoublingOfVariables} holds and  $\gamma_\lambda \in C_0^\infty((0,T) \times (0,T))$. Differentiating \eqref{eq:gamma-lambda}, we find
    \[
        (\partial_t + \partial_s)\gamma_\lambda(t,s) = \frac{1}{\lambda}\phi\Bigl( \frac{t-s}{\lambda} \Bigr)\,\gamma'\Bigl( \frac{t+s}{2} \Bigr).
    \]
    Here the problematic terms with the derivative $\phi'$ have cancelled.
  With this choice inequality  \eqref{eq:DoublingOfVariables} becomes
    \begin{equation*}
       \int\limits_0^T \!\!\! \int\limits_0^T \!\!\! \int\limits_\Omega \left(b(u_1(t)) - b(u_2(s))\right)^+ \frac{1}{\lambda}\phi\Bigl( \frac{t-s}{\lambda} \Bigr)\gamma'\Bigl( \frac{t+s}{2} \Bigr) \d x \d t \d s \,
      \geq\,0.
      \end{equation*}
     By a standard procedure for approximating Dirac's delta with mollifiers, by sending $\lambda \to 0$ we arrive at
    \[
        \int\limits_0^T \!\! \int\limits_\Omega \left(b(u_1(t)) - b(u_2(t))\right)^+\gamma'\left( t \right)\d x \d t\,\, \geq\,\, 0.
    \]
  See Lemma 2.9 on page 76 in \cite{holden15} for details.  This completes the proof.
\end{proof}

\begin{proof}[Proof of Theorem \ref{thm:Comparison-Princple}]

  Let $0<t_1<t_2<T$. By replacing $\gamma$ in \eqref{eq:lambda->0} by a suitable cut-off function, approximating the characteristic function of the interval $[t_1,t_2]$, we obtain at the limit 
      \[
       \int\limits_\Omega \left(b(u_1(t_2)) - b(u_2(t_2))\right)^+ \d x\,\leq  \int\limits_\Omega \left(b(u_1(t_1)) - b(u_2(t_1))\right)^+ \d x.
    \]
    By this and the sub-additivity of $z \mapsto z^+$, we get
    \begin{align*}
\int\limits_\Omega \left(b(u_1(t_2)) - b(u_2(t_2))\right)^+\! \d x\le& \int\limits_\Omega \left(b(u_1(t_1))-b(\psi_1(t_1))\right)^+\!\d x \\
&+\int\limits_\Omega \left(b(\psi_1(t_1))-b(\psi_2(t_1))\right)^+\! \d x\\
&+\int\limits_\Omega \left(b(\psi_2(t_1))-b(u_2(t_1))\right)^+\! \d x.
\end{align*}
The first and the third terms on the right-hand side  converge  to zero as $t_1\to 0$, and  the second term is zero by \eqref{eq:psi-epsilon}. Thus
\begin{align*}
\int\limits_\Omega \bigl(b(u_1(t_2)) - b(u_2(t_2))\bigr)^+ \d x\,\,\leq\,\, 0,
\end{align*}
and from which we conclude that $u_1(x,t_2) \leq u_2(x,t_2).$ Since $t_2$ is arbitrary,
 this completes the proof of the comparison principle.
\end{proof}

In some cases we can allow $\varepsilon = 0$. Constant lateral boundary values are included in the next result.

\begin{proposition}[\textsf{Monotone Lateral Values}]
\label{prop:monotone} Let $1<p<\infty.$ Suppose that the lateral boundary values are monotone, say
  $$\psi(x,t)\leq \psi(x,s)\quad\text{when}\quad 0<t<s<T,\,\,x\in \partial \Omega.$$ Then the solution in Definition \ref{def:PDE} is unique.
\end{proposition}

\begin{proof} We can use the same proof, if we make sure that $s\geq t$, so that
  $\psi(x,s) \geq \psi(x,t).$  Replace $\gamma_{\lambda}$ in (\ref{eq:gamma-lambda}) by the modified test function
  $$\gamma^+_{\lambda}(t,s) = \frac{1}{\lambda}\,\phi\Bigl(\frac{s-t}{\lambda} - 1\Bigr)\,\gamma\Bigl(\frac{t+s}{2}\Bigr).$$
  This choice guarantees that the integrands vanish outside the interval $t < s < t + 2\lambda$. Again, the function
  $$f(x,t)\,=\, \eta'_{\sigma,+}\!\bigl(u_1(x,t)-u_2(x,s)\bigr)\,\gamma_{\lambda}^+(t,s)$$
  belongs to $L^p(0,T;W_0^{1,p}(\Omega)).$
  \end{proof}

\section{Proof of the uniqueness}
In this section we prove the main result of the paper, uniqueness with sign-changing $C^2$ boundary data.
Our proof relies on properties of suitably regular auxiliary solutions.
Given $\psi \in C^2(\overline{\Omega_T})$, we say that a weak solution $u \in C(\overline{\Omega_T})$ with boundary values $\psi$ is \emph{$t$-regular} if the Sobolev time derivative $\partial_t (\abs{u}^{p-2}u)$ exists and
\[
    \int\limits_0^T \!\! \int\limits_\Omega (\abs{\partial_t b(u)}^2 + \abs{\nabla u}^{p}) \d x \d t\, < \,\infty.
\]
We now note that there exists one and only one such $t$-regular solution, see \cite[Pages 4-6]{lindqvistps} for details. It is unique at least within their own regularity class. The next theorem guarantees that there are no other weak  solutions
\begin{theorem}\label{thm:Main} Let $p \geq 2.$
    A solution with  boundary values $\psi \in C^2(\overline{\Omega_T})$ is unique.
\end{theorem}

\begin{proof}
   Let $u$ be an \emph{arbitrary} solution with the given boundary values $\psi$. According to \cite[Theorem 5]{lindqvistps}, there always exists a $t$-regular solution with the boundary values $\psi\in C^2(\overline{\Omega_T})$. Moreover, it is well-known that such a $t$-regular solution is unique, see for example \cite[Corollary 3]{lindqvistps}, although the existence of less regular solutions is not yet excluded. Given a constant $\varepsilon > 0$,   
   there exist unique $t$-regular solutions
\begin{equation*}\begin{cases}
    u_{+\varepsilon}\quad\text{with boundary values}\quad\psi + \eps\\
   \, u_{0}\quad\,\,\text{with boundary values}\quad\psi + 0\\
    u_{-\varepsilon}\quad\text{with boundary values}\quad\psi - \eps.
\end{cases}\end{equation*}
On the other hand, by the comparison principle in Theorem \ref{thm:Comparison-Princple}, we have for almost every $(x,t) \in \Omega \times (0,T)$ that
    \begin{equation}\label{eq:Compare}
      u_{-\varepsilon}(x,t) \leq u(x,t) \leq u_{+\varepsilon}(x,t).
    \end{equation}
    It is essential that the above comparison principle does not require $u$ to be $t$-regular. We claim that $u = u_0$.
    
    Finally, since
    $\psi \in C^1(\overline{\Omega_T})$ and $\|\partial_t \nabla \psi\|_{L^p(\Omega_T)}  < \infty$ by our assumption, we may use  \cite[Theorem 5 and Corollary 3]{lindqvistps} to conclude that, upon
    passing to a subsequence if necessary, 
    \begin{align*}
u_{\pm\varepsilon}\, \to\, u_0\quad \text{in}\quad  L^p(0,T;W^{1,p}(\Omega))\quad \text{as} \quad \varepsilon \to 0.
    \end{align*}
    The point is that both sequences have the same limit, viz.\ $u_0$.
   For  a  subsequence the convergences  hold pointwise for almost every $(x,t) \in \Omega \times (0,T)$.
      This together with \eqref{eq:Compare} implies that $u=u_0$,  and thus the proof of the uniqueness is complete.
\end{proof}

Our convenient assumption $\psi \in C^2(\overline{\Omega_T})$ comes from the convergence argument in the proof above. To be more precise, it is based upon the \emph{finite} majorant in  \cite[Theorem 5]{lindqvistps} 
\begin{align*}
     \int\limits_0^{T}\!\!\! \int\limits_{\Omega}&\Bigl(\Big\vert\frac{\dd}{\dd t}\bigl(|u|^{\frac{p-2}{2}}u\bigr)\Big\vert^2 +|\nabla u|^p + |u|^p\Bigr) \d x \d t  +\int\limits_{\Omega}|\nabla u(x,T)|^p\d x\\
  &\le C\int\limits_{\Omega}\left(|\psi(x,0)|^p +|\nabla \psi(x,0)|^p\right)\d x \\
&\hspace{5 em}  +C\int\limits_0^T\!\!\!\int\limits_{\Omega}\Bigl(|\psi|^p+|\nabla \psi|^p+\Bigl|\frac{\partial\psi}{\partial t}\Bigr|^p+\Bigl|\frac{\partial \,\,}{\partial t}\nabla \psi \Bigr|^p\Bigr)\d x \d t.
\end{align*}
The proof also has an interesting consequence. Since we obtained $u = u_0$, we can read off the following  result from the above estimate. \enlargethispage{2\baselineskip}
\begin{corollary} A weak solution with $C^2$ boundary data possesses the Sobolev derivatives
  \begin{align*} 
      &\partial_t(|u|^{\frac{p-2}{2}}u)\,\in L^2(\Omega_T),\qquad p \geq 2.
    \end{align*}
\end{corollary}
\def\cprime{$'$}

\end{document}